\documentclass[12pt]{amsart}

\usepackage{amssymb,euscript,amsmath, mathrsfs,amscd}
\usepackage[dvips]{graphicx}
\usepackage[dvips]{color}
\usepackage{epsfig,pstricks}
\usepackage{fullpage}

\newcounter{ENUM}
\newcommand{\beas}{\begin{eqnarray*}}
\newcommand{\eeas}{\end{eqnarray*}}
\newcommand{\itm}{\item}
\newenvironment{ilist}{\renewcommand{\theENUM}{\roman{ENUM}}\renewcommand{\itm}{\addtocounter{ENUM}{1}\item[(\theENUM)]}\begin{itemize}\setcounter{ENUM}{0}}{\end{itemize}}
\newenvironment{alist}[1][0]{\renewcommand{\theENUM}{\alph{ENUM}}\renewcommand{\itm}{\addtocounter{ENUM}{1}\item[\theENUM)]}\begin{itemize}\setcounter{ENUM}{#1}}{\end{itemize}}

\newtheorem{thm}{Theorem}[section]
\newtheorem{prop}[thm]{Proposition}
\newtheorem{lem}[thm]{Lemma}
\newtheorem{cor}[thm]{Corollary}

\theoremstyle{definition}
\newtheorem{defn}[thm]{Definition}
\newtheorem{ques}[thm]{Question}

\theoremstyle{remark}
\newtheorem{notn}[thm]{Notation}
\newtheorem{rem}[thm]{Remark}

\numberwithin{equation}{section}

\def\N{{\mathbb N}}
\def\Z{{\mathbb Z}}
\def\R{{\mathbb R}}
\def\S{{\mathfrak S}}

\def\ds{\displaystyle}
\newcommand{\bm}[1]{{\boldsymbol{#1}}}
\def\0{\bm{0}}
\def\k{\bm{k}}
\def\q{\bm{q}}
\def\r{\bm{r}}
\def\s{\bm{s}}
\def\v{\bm{v}}
\def\u{\bm{u}}
\def\w{\bm{w}}
\def\x{\bm{x}}
\def\y{\bm{y}}
\def\z{\bm{z}}

\def\ts{{\tilde{\bm s}}}

\def\tu{{\tilde{\bm u}}}

\def\Asc{\operatorname{Asc}}
\def\asc{\operatorname{asc}}
\def\Des{\operatorname{Des}}
\def\des{\operatorname{des}}
\def\rev{\operatorname{reverse}}
\def\KR{\operatorname{KR}}
\def\Par{\operatorname{Par}}
\def\REM{\operatorname{REM}}
\def\bREM{\overline{\operatorname{REM}}}

\def\L{\mathcal L}
\def\l{\ell}

\subjclass[2010]{52B20, 05A17}

\begin{document}
\title{The Lecture Hall Parallelepiped}
\author{Fu Liu}
\thanks{Fu Liu is partially supported by the Hellman Fellowship from
  UC Davis.} \address{Fu Liu, Department of Mathematics, University of
  California, Davis, One Shields Avenue, Davis, CA 95616 USA.}
\email{fuliu@math.ucdavis.edu}

\author{Richard P. Stanley}
\thanks{Richard Stanley is partially supported by NSF grant DMS-1068625.}
\address{Richard Stanley, Department of Mathematics, M.I.T.,
  Cambridge, MA 02139 USA.}
\email{rstan@math.mit.edu}

\begin{abstract}
The $\bm{s}$-lecture hall polytopes $P_\s$ are a class of integer
polytopes defined by Savage and Schuster which are closely related to
the lecture hall partitions of Eriksson and Bousquet-M\'elou. We
define a half-open parallelopiped Par$_\s$ associated with $P_\s$ and
give a simple description of its integer points. We use this
description to recover earlier results of Savage et al.\ on the $\delta$-vector (or $h^*$-vector) and to obtain the connections to $\s$-ascents and $\s$-descents, as well as some generalizations of these results. 

\end{abstract}

\maketitle

\section{Introduction}
Suppose that $P$ is an $n$-dimensional integral polytope, i.e., a
(convex) polytope whose vertices have integer coordinates. Let
$i(P,t)$ be the number of lattice points in the $t$th dilation $tP$ of
$P$. Then $i(P,t)$ is a polynomial in $t$ of degree $n$, called the
{\it Ehrhart polynomial} of $P$ \cite{ehrhart}. One way to study the
Ehrhart polynomial of an integral polytope is to consider its
generating function $\sum_{t \ge 0} i(P,t) z^t.$ It is known that the
generating function has the form 
\[ \sum_{t \ge 0} i(P,t) z^t = \frac{\delta_P(z)}{(1-z)^{n+1}},\]
where $\delta_P(z)$ is a polynomial of degree at most
$n$ with nonnegative integer coefficients \cite{Stan1980}. We denote
by $\delta_{P,i}$ the coefficient of $z^i$ in $\delta_P(z)$, for $0
\le i \le n.$ Thus $\delta_P(z) = \sum_{i=0}^n \delta_{P,i} z^i.$ 
For an $n$-dimensional polytope $P$ in $\R^n$, the {\it normalized
  volume} nvol$(P)$ is given by nvol$(P)=n!\cdot\mathrm{vol}(P)$,
where vol$(P)$ is the usual volume (Lebesgue measure).
Another well-known result is that $\delta_P(1)=\sum_{i=0}^n
\delta_{P,i}$ is the 
normalized volume of $P.$ We call $(\delta_{P,0}, \delta_{P,1},
\dots, \delta_{P,n})$ the {\it $\delta$-vector} or {\it $h^*$-vector}
of $P$. In this paper, we will investigate the $\delta$-vectors of
$\s$-lecture hall polytopes, which were introduced by Savage and
Schuster \cite{SavSch2012}. A basic idea we use is a result by the second author \cite[Lemma~4.5.7]{stanleyec2}: one can determine the $\delta$-vector of an integral simplex by counting the number of lattice points inside an associated parallelepiped. 

Let $\s = (s_1, \dots, s_n)$ be a sequence of positive integers. An
{\it $\s$-lecture hall partition} is an integer sequence $\lambda =
(\lambda_1, \lambda_2, \dots, \lambda_n)$ satisfying 
\[ 0 \le \frac{\lambda_1}{s_1} \le \frac{\lambda_2}{s_2} \le \cdots
\le \frac{\lambda_n}{s_n}.\] 
When $\s = (1, 2, \dots, n),$ this gives the original lecture hall
partitions introduced by Bousquet-M\'elou and Eriksson 
\cite{BousErik1997}. Savage and Schuster \cite{SavSch2012} define the
{\it $s$-lecture hall polytope} to be the polytope, denoted $P_\s$,
in $\R^n$ defined by the inequalities
\[ 0 \le \frac{x_1}{s_1} \le \frac{x_2}{s_2} \le \cdots \le
\frac{x_n}{s_n} \le 1.\] 
They use the $\s$-lecture hall polytopes to establish a connection
between $\s$-lecture hall partitions and their geometric setup. 
A further result, appearing in \cite{CorLeeSav2005}, is 
that the Ehrhart polynomial of the lecture hall polytope
associated to $\s = (1, 2, \dots, n)$ and the anti-lecture hall polytope associated to $s=(n,n-1,\dots,1)$ is the same as that of the
$n$-dimensional unit cube. 
It is well known that components in the
$\delta$-vector of the $n$-dimensional unit cube are Eulerian numbers,
which count the number of permutations in $\S_d$ with a certain number
of descents \cite[Prop.~1.4.4]{stanleyec1}. Thus, the same is true for 
$P_{(1,2\dots,n)}$ and $P_{(n,n-1,\dots,1)}.$  

It is easy to see that $P_\s$ has the vertex set
 \[ \{(0,0,0,\dots, 0), (0,0, \dots, 0, s_n), (0,0,\dots,0, s_{n-1},
  s_n), \dots, (s_1, s_2, \dots, s_n) \}.\] 
Hence $P_\s$ is a simplex with normalized volume $\prod_{i=1}^n s_i.$
In particular, when $\s \in \S_n,$ the normalized volume is $n!$,
which is exactly the cardinality of $\S_n.$ Thus, the sum of the
components in the $\delta$-vector of $P_\s$ is $\prod_{i=1}^n s_i$ or
$n!.$ On the other hand, since $P_\s$ is a simplex, its
$\delta$-vector corresponds to gradings of the lattice points in a
fundamental parallelepiped associated to it. (See Lemma \ref{lem:cnnt}
for details.)

The original motivation of this paper was to give a bijection between
$\S_n$ and lattice points in the fundamental parallelepipeds associated
to $P_{(1,2,\dots,n)}$ and $P_{(n,n-1,\dots,1)}$ so that we can recover the result of Corteel-Lee-Savage \cite{CorLeeSav2005} on the Ehrhart polynomials of these polytopes. 
In fact we can extend
our original aim to the fundamental parallelopiped associated to
$P_\s$ for any sequence $\s$ of positive integers. Our results are
stated in terms of ascents and descents of certain sequences
associated to $\s$ which generalize the notion of the inversion
sequence of a permutation.
We also consider the connection between descents and ascents of sequences associated to $\s$ and the reverse of $\s.$

The paper is organized as follows. In Section \ref{sec:background}, we review basic results of $\delta$-vectors that are relevant to our paper, introduce the $\s$-lecture hall parallelepiped $\Par_\s$, and establish in Lemma \ref{lem:cnnt} the connection between the number lattice points in $\Par_\s$ and the $\delta$-vector of $P_\s.$ In Section \ref{sec:bijs}, we give a bijection $\REM_\s$ from the lattice points in $\Par_\s$ to some simple set (which we call $\Psi_n$). By figuring out the inverse of $\REM_\s,$ we are able to describe in Theorem \ref{thm:delta-des} the $\delta$-vector of $P_{\s^*},$ a polytope closely related to $P_\s,$ using the language of descents. A special situation of this theorem agrees with results by Savage-Schuster \cite{SavSch2012} on the $\delta$-vector of $P_\s$. In Section \ref{sec:anti}, we apply results from Section \ref{sec:bijs} to the case when $\s=(n,n-1,\dots,1)$ and recover the result of Corteel-Lee-Savage on the Ehrhart polynomial of the anti-lecture hall polytope. In Sections \ref{sec:reversal}, we consider $\s$ and its reversal $\u,$ and their corresponding polytopes $P_\s$ and $P_\u,$ and provide a bijection from the lattice points $\Par_{\s^*}$ to the lattice points in $\Par_{\u^*}$ through maps defined in Section \ref{sec:bijs}. In Section \ref{sec:s1=1}, we show that the $\delta$-vector of $P_\s$ can be described using ascents of elements in $\Phi_n,$ using which we give the desired bijective proof for Corteel-Lee-Savage's result on the Ehrhart polynomial of $P_{(1,2,\dots,n)}.$

\section{Background}\label{sec:background}
For any nonnegative integer $N,$ we denote by $\langle N\rangle$ the
set $\{0, 1, \dots, N\}.$
\subsection{Descents, the $\delta$-vector of a unit cube, etc.} 

Write $\N=\{0,1,2.\dots\}$. Let $\r=(r_1,\dots,r_n) \in \N^n.$ We say
that $i$ is a {\it (regular) descent} of $\r$ if $r_i > 
r_{i+1}.$ Define the {\it descent set} $\Des(\r)$ of $\r$ by 
 \[ \Des(\r) = \{i \ | \ r_i > r_{i+1}\},\] 
and define its size $\des(\r)= \# \Des(\r).$

The {\it Eulerian number} $A(n,i)$ is the number of permutations $\pi
\in \S_n$ with exactly $i-1$ descents. Let $\square_n$ denote the
$n$-dimensional unit cube. Then the $\delta$-vector of $\square_n$ is
given by  
 \[ \delta_{\square_n, i}  = A(n,i+1) = \# \{ \pi \in \S_n \ | \
  \des(\pi) = i \}.\] 
By \cite[Corollary 1]{SavSch2012} we have that for $\s=(1,2,\dots,n),$ 
\[ \delta_{P_\s, i}  = A(n,i+1) = \# \{ \pi \in \S_n \ | \ \des(\pi) = i \}.\]

There are many other statistics related to permutations also counted
by Eulerian numbers. Given a permutation $\pi =(\pi_1,\dots,\pi_n) \in
\S_n,$ a pair $(\pi_j,\pi_k)$ is an {\it inversion} of $\pi,$ if $j <
k$ and $\pi_j > \pi_k.$ Define $I(\pi) = (a_1, \dots, a_n)$, where
$a_i$ is the number of inversions $(\pi_j,\pi_k)$ of $\pi$ that ends
with $i=\pi_k$.  The sequence $I(\pi)$ is known as the {\it inversion
  sequence} or {\it inversion table} of the permutation $\pi.$
Clearly, $I(\pi) \in \langle n-1\rangle \times \cdots \times \langle
1\rangle \times \langle 0\rangle.$ In fact, $I: \S_n \to \langle
n-1\rangle \times \cdots \times \langle 1\rangle \times \langle
0\rangle$ is a bijection \cite[Prop.~1.3.12]{stanleyec1}.  In this
paper, it is more convenient to use inversion sequences to represent
permutations. We give statistics of inversion sequences that are
counted by Eulerian numbers.

\begin{lem}\label{lem:invseq}
The number of inversion sequences of length $n$ with $i$ descents is
the Eulerian number $A(n, i+1).$ 
\end{lem}

\begin{proof}
Suppose that $\r=(r_1,\dots,r_n)$ is the inversion sequence of $\pi
\in \S_n.$ Then 
\begin{eqnarray*}
	\text{$i$ is a descent of $\r$, i.e., $r_i > r_{i+1}$}
        &\Longleftrightarrow& \text{$i+1$ precedes $i$ in $\pi$} \\ 
	&\Longleftrightarrow& \text{$i$ is a descent of $\pi^{-1}$.}
\end{eqnarray*}
The lemma follows from the fact that $\pi \mapsto \pi^{-1}$ is a
bijection on $\S_n.$ 
\end{proof}

\subsection{$\delta$-vector of simplices}
 When $P$ is a simplex, it is easy to describe its $\delta$-vector. We
 first give some related definitions and notation. 

For a set of independent vectors $W=\{\w_1, \dots, \w_n\},$ we
denote by $\Par(W) = \Par(\w_1,\dots,\w_n)$ the {\it fundamental
(half-open) parallelepiped generated by $W$}: 
\[\Par(\w_1,\dots, \w_n) := \left\{  \sum_{i=1}^n c_i \w_i \ | \  0
  \le c_i < 1 \right\}.\] 

For any set $S \subset \Z^N,$ we denote by $\L^i(S)$ the set of
lattice points in $S$ whose last coordinates are $i:$ 
\[ \L^i(S) :=  \left\{ \x \in S \cap \Z^{N} \ | \ \text{last
    coordinate of $\x$ is $i$} \right\}, \] 
and let $ \l^i(S) := \# \L^i(S)$ be the cardinality of $\L^i(S).$ 

For convenience, for any vector $\v \in \R^N$, we let $\v^* := (\v,1)$
be the vector obtained by appending $1$ to the end of $\v.$ 

Suppose $P$ is an $n$-dimensional simplex with vertices $\v_0, \v_1,
\dots, \v_n$. Then the $\delta$-vector of $P$ is determined by the
grading of the fundamental parallelepiped $\Par(\v_0^*, \dots,
\v_n^*)$. More precisely \cite[Lemma~4.5.7]{stanleyec2}, 
\begin{equation}\label{equ:detdelta}
\delta_{P,i} = \l^i(\Par(\v_0^*,\dots, \v_n^*)), \quad 0
\le i \le n. 
\end{equation}
(Note that $\l^i(\Par(\v_0^*,\dots, \v_n^*)) = 0$ for all $i > n.$)

\subsection{$\s$-lecture hall parallelepiped}
\begin{defn}
Given a sequence $\s = (s_1, \dots, s_n)$ of positive integers, the
{\it $\s$-lecture hall parallelepiped}, denoted by $\Par_\s,$ is the
fundamental parallelepiped generated by the non-origin vertices of the
$\s$-lecture hall polytope $P_\s:$ 
	\[\Par_\s := \Par((0,0, \dots, 0, s_n), (0,0,\dots,0, s_{n-1},
        s_n), \dots, (s_1, s_2, \dots, s_n)).\] 
\end{defn}

\begin{lem}\label{lem:cnnt}
Suppose that $\s = (s_1, \dots, s_n)$ is a sequence of positive
integers. Then the $\delta$-vector of $P_\s$ is given by
\begin{equation}\label{equ:delta}
  \delta_{P_\s,i} = \l^i(\Par_{\s^*}), \quad 0 \le i \le n. 
\end{equation}
Furthermore, if $s_n = 1,$ then the two fundamental parallelepipeds
$\Par_{\s}$ and $\Par_{\s^*}$ have the same grading: 
\begin{equation}\label{equ:same}
\l^i(\Par_{\s}) = \l^i(\Par_{\s^*}), \quad 0 \le i \le n.
\end{equation}
Hence, 
\begin{equation}\label{equ:delta1}
\delta_{P_\s,i} = \begin{cases}\l^i(\Par_{\s}), & 0 \le i \le n-1;
  \\ 0, & i=n.\end{cases} 
\end{equation}
\end{lem}

\begin{proof}
Formula \eqref{equ:delta} follows from \eqref{equ:detdelta}
and the observation that  
  \[ \Par_{\s^*} = \Par( \v^* \ | \ \v \text{ is a vertex of $P$}).\]

Suppose that $s_n=1.$ We claim that for any $\x \in \Par_{\s^*}\cap
\Z^{n+1}$, the last two coordinates of $\x$ are the same. Let
$\x=(x_1,\dots, x_{n+1}) \in \Par_{\s^*} \cap\Z^{n+1}.$ There exist
(unique) $c_1,\dots, c_{n}, c_{n+1} \in [0,1)$ such that 
\begin{eqnarray*}
	\x &=& c_{n+1} (0,\dots, 0, 0 ,0, 1) + c_n (0, \dots, 0, 1,
        1) + c_{n-1} (0, \dots, s_{n-1}, 1, 1) \\
	& & + \cdots + c_1 (s_1,\dots, s_{n-1}, 1, 1).
\end{eqnarray*}
Then 
  \[ x_n = \sum_{i=1}^n c_i, \text{ and } x_{n+1} =
     \sum_{i=1}^{n+1} c_i.\] 
Since both $x_n$ and $x_{n+1}$ are integers, the number $c_{n+1} =
x_{n+1}-x_n$ is an integer. However, $0 \le c_{n+1} < 1.$ We must have
that $c_{n+1}=0.$ Therefore $x_n=x_{n+1}$, so the claim holds.  
	
By our claim, one sees that the map that drops the last coordinate of
a point in $\R^{n+1}$ induces a bijection between $\L^i(\Par_{\s^*})$
and $\L^i(\Par_\s)$ for every $i.$ Hence equation~\eqref{equ:same}
follows.  
	
Finally, the last coordinate of any point in $\Par_{\s}$ is strictly
smaller than $n s_n = n.$ Hence, $\L^n(\Par_{\s}) = \emptyset$ and
$\l^n(\Par_\s) =0.$ Then formula \eqref{equ:delta1} follows from
\eqref{equ:delta} and \eqref{equ:same}. 
\end{proof}

\section{Bijections}\label{sec:bijs}

Throughout this section, we assume that $\s=(s_1,\dots,s_n)$ is a
sequence of positive integers.  For brevity, for the rest of the paper, whenever $\s$ is a fixed sequence, we associate the following set to $\s$:
\begin{equation*}
	\Psi_n  =  \langle s_1-1\rangle \times \cdots \times \langle s_n-1 \rangle. 
  \end{equation*}
This set coincides with the set $I_n$ of $\s$-inversion sequences 
introduced in \cite{SavSch2012} and further investigated in \cite{SavVis2012,PenSav2012a, PenSav2012b}.
  \begin{defn}\label{defn:REM} 
We define a map
\[ \REM_\s: \Par_\s \cap \Z^n \to \Psi_n  \] 
in the following way.
Let $\x = (x_1, \dots, x_n) \in \Par_\s \cap \Z^n.$ For each $x_i,$
let $k_i = \lfloor \frac{x_i}{s_i} \rfloor$ be the quotient of
dividing $x_i$ by $s_i,$ and $r_i$ be the remainder. Hence
 \[ x_i = k_i s_i +
    r_i,\] 
where $k_i \in \langle n-1\rangle $ and $r_i \in \langle s_i-1\rangle
.$ Let $\k = (k_1,\dots, k_n)$ and $\r = (r_1, \dots, r_n).$ Then we
define $\REM_\s(\x) = \r$.
\end{defn}

\begin{lem}\label{lem:REMbij}
$\REM_\s$ is a bijection from $\Par_\s \cap \Z^n$ to $\Psi_n$. 
\end{lem}

In order to prove Lemma \ref{lem:REMbij}, we write $\REM_\s$ as a
composition of two maps. Let    
\[ f_\s: \Par_\s \cap \Z^n \to \langle n-1\rangle ^n \times \Psi_n, \quad
\x \mapsto (\k,\r), \]  
where $\k$ and $\r$ are defined as in Definition~\ref{defn:REM}.
We denote by $\KR_\s$ the image set of $\Par_\s \cap \Z^n$ under the
map $f_\s.$ 
It is clear that the map $f_\s$ is a bijection between $\Par_\s \cap \Z^n$ and $\KR_\s.$  

Let \[ g_\s: \KR_\s \to \Psi_n, \quad (\k,\r) \mapsto \r.\] 
Clearly, $\REM_\s$ is the composition of $f_\s$ and $g_\s,$ and Lemma
\ref{lem:REMbij} follows from the following lemma. 

\begin{lem}\label{lem:gbij}
The map $g_\s$ gives a bijection between $\KR_\s$ and $\Psi_n$. 
\end{lem}


To prove Lemma \ref{lem:gbij}, we will construct an inverse for
$g_\s;$ in other words, we will show how to recover the quotient
vector $\k$ from the remainder vector $\r.$ We give the following
preliminary definition and lemma.

\begin{defn}
Let $\r \in \N^n.$ We say that $i$ is an {\it $\s$-descent of
  $\r$} if $\displaystyle \frac{r_i}{s_i} > \frac{r_{i+1}}{s_{i+1}}.$  
	
We denote by $\Des_\s(\r)$ the set of $\s$-descents of $\r,$ and let
$\des_\s(\r) = \# \Des_\s(\r)$ be its cardinality. 
For any $1 \le i \le n$, we let $\Des_\s^{<i}(\r)$ be the set of
$\s$-descents of $\r$ whose indices are strictly smaller than $i:$ 
	\[ \Des_\s^{<i}(\r) = \{ j < i \ | \ j \text{ is an
          $\s$-descent of $\r$}\},\] 
and $\des_\s^{<i}(\r) = \# \Des_\s^{<i}(\r)$ be its cardinality.

We similarly define $\Des^{<i}$ and $\des^{<i}$ for (regular) descents.
\end{defn}

\begin{lem}\label{lem:char}	\begin{alist}
	\itm A point $\x = (x_1,\dots, x_n)$ is in $\Par_\s$ if and
  only if
  \[ 0 \le \frac{x_1}{s_1} < 1, \text{ and } 0 \le
        \frac{x_{i+1}}{s_{i+1}} - \frac{x_i}{s_i} < 1, \quad 
        1 \le i \le n-1.\]  
\itm Let $\r \in \Psi_n$.
 Then a point $(\k,\r) = ( (k_1,\dots,k_n), \r)$ is in $\KR_\s$ if and
 only if $k_1=0$ and for any $i \in \{1,\dots,n-1\},$ 
 \begin{equation}\label{equ:kdiff}
	 k_{i+1}-k_i = \begin{cases}1, & \text{if $i$ is an
       $\s$-descent of $\r$;} \\ 0, & \text{otherwise.}\end{cases} 
	 \end{equation}

	\end{alist}
\end{lem}

\begin{proof}
First, $\x \in \Par_\s$ if and only if there exists $c_1, \dots, c_n$
in $[0,1)$ such that  
	\begin{eqnarray*}
	\x &=& c_n (0, \dots, 0, s_n) + c_{n-1} (0, \dots, 0,
        s_{n-1}, s_n) + \cdots + c_1 (s_1, \dots, s_n)  \\ 
       &=& ( s_1 c_1, \ s_2 (c_1+c_2), \ \dots, \ s_n
        (c_1+c_2+ \cdots + c_n)). 
	\end{eqnarray*}
This is equivalent to the existence of $c_1, \dots, c_n \in [0,1)$
  such that  
 \begin{eqnarray*} \ds\frac{x_1}{s_1} & = & c_1 \\ \ds\frac{x_2}{s_2}
   & = & c_1 + c_2 \\ & \vdots & \\ \frac{x_{n}}{s_n} & = &  c_1 + c_2
   + \cdots + c_n \end{eqnarray*}  
Solving the above equations for $c_i$'s, we see that a) follows.

To prove b), we let $x_i = k_i s_i + r_i$ for each $i.$ Note that
$(\k,\r)\in\KR_\s$ if and only if $\x:=(x_1,\dots,x_n) \in \Par_\s
\cap \Z^n,$ which is equivalent to $\k \in \Z^n$ and $\x \in \Par_\s.$
Applying part a) to $\x,$ we get that $(\k, \r) \in \KR_\s$ if and
only if $\k \in \Z^n$ and 
\[ 0 \le \frac{k_1 s_1 + r_1}{s_1} < 1, \text{ and }  0 \le
\frac{k_{i+1} s_{i+1} + r_{i+1}}{s_{i+1}} - \frac{k_i s_i + r_i}{s_i}
< 1, \quad 1 \le i \le n-1.\] 
The above inequalities are equivalent to
\[ 0 \le k_1 + \frac{r_1}{s_1} < 1, \text{ and }  0 \le k_{i+1}-k_i +
\frac{r_{i+1}}{s_{i+1}} - \frac{r_i}{s_i} < 1, \quad 1 \le
i \le n-1.\] 
Note that $0 \le \frac{r_1}{s_1}<1$ and $-1 < \frac{r_{i+1}}{s_{i+1}}
- \frac{r_i}{s_i} < 1$. One checks that 
\[ k_1 \in \Z \text{ and } 0 \le k_1 + \frac{r_1}{s_1} < 1
\ \Longleftrightarrow \ k_1 = 0,\] 
and for any $1 \le i \le n-1$ given $k_i \in \Z,$  
\[ k_{i+1} \in \Z \text{ and }  0 \le k_{i+1}-k_i +
\frac{r_{i+1}}{s_{i+1}} - \frac{r_i}{s_i} < 1 \ \Longleftrightarrow
k_{i+1}-k_i \text{ is given as in \eqref{equ:kdiff}.}\] 
Therefore, we have b).
\end{proof}

Part b) of Lemma \ref{lem:char} provides us a way to construct the
inverse of $g_\s.$ 
For any $\r \in \Psi_n$, we define $h_\s(\r) = (\k,\r)=( (k_1,\dots,k_n),
\r)$, where 
\[ k_i = \des_\s^{<i}(\r), \quad 1 \le i \le n.\]
By Lemma \ref{lem:char}(b) we see that $h_\s$ is the inverse of
$g_\s.$ Hence, we have proved Lemma \ref{lem:gbij}. Our discussion
also gives us the inverse map of $\REM_\s.$ 

\begin{thm}\label{thm:REMinv}
The inverse of the map $\REM_\s$ (defined in Definition
\ref{defn:REM}) is: 
{\small	\begin{eqnarray*}
   \REM_\s^{-1}: \Psi_n &\to& \Par_\s \cap \Z^n \\ 
	\r= (r_1, \dots, r_n) &\mapsto& (\des_\s^{<1}(\r) s_1 + r_1,
        \dots, \des_\s^{<n}(\r) s_n + r_n)  
\end{eqnarray*}}
\end{thm}

Note that $\des_\s^{<n}(\r) = \des_\s(\r),$ and thus when $s_n$=1, 
\[ \des_\s^{<n}(\r) s_n + r_n = \des_\s(\r).\]
Hence we have the following result.

\begin{cor}\label{cor:sn=1}
Suppose $s_n = 1.$ Then
\begin{eqnarray*}
  \L^i(\Par_\s) &=&  \{ \x \in \Par_\s \cap \Z^n \ |
     \ \des_\s(\REM_\s(\x)) = i\} \\ 
    &=& \{ \REM_\s^{-1}(\r) \ | \ \des_\s(\r) = i, \ \r \in
    \Psi_n \},  
\end{eqnarray*}
and
\[ \l^i(\Par_\s) = \# \{ \r \in \Psi_n \ | \ \des_\s(\r) = i\}.\]  
\end{cor}

Applying this to $s^*$ whose last coordinate is $1$ by definition, we
get the next corollary. 
\begin{cor}\label{cor:s*}
\begin{eqnarray*}
 \L^i(\Par_{\s^*}) &=&  \{ \x \in \Par_{\s^*} \cap \Z^{n+1} \ |
 \ \des_{\s^*}(\REM_{\s^*}(\x)) = i\} \\ 
	&=& \{ \REM_{\s^*}^{-1}(\r) \ | \ \des_{\s^*}(\r) = i, \ \r
        \in \Psi_n  \times \langle 0 \rangle \}, 
\end{eqnarray*}
and
\[ \l^i(\Par_{\s^*}) = \# \{ \r \in \Psi_n  \times \langle 0 \rangle  \ | 
\ \des_{\s^*}(\r) = i\}.\]  
\end{cor}

The above two corollaries, together with Lemma \ref{lem:cnnt}, give
the following result on $\delta$-vectors of the $\s$-lecture hall
polytope. 
\begin{thm}\label{thm:delta-des}
Suppose that $\s = (s_1, \dots, s_n)$ is a sequence of positive
integers. Then the $\delta$-vector of the $\s$-lecture hall polytope
$P_\s$ is given by 
\begin{equation}\label{equ:delta-des}
	\delta_{P_{\s^*},i} = \# \{ \r \in \Psi_n\times \langle 0\rangle \ |
        \ \des_{\s^*}(\r) = i\},\quad 0 \le i \le n. 
\end{equation}
Furthermore, if $s_n = 1$ then
\begin{equation}\label{equ:delta-des1}
\delta_{P_\s,i} = \# \{ \r \in \Psi_n \ | \ \des_{\s}(\r) = i\}, \quad 0
\le i \le n. 
\end{equation}
\end{thm}

We note that equation \eqref{equ:delta-des1} agrees with Corollary 4 in \cite{SavSch2012}.

The bijection $\REM$ is not always the most convenient one to
use. Fortunately, there are many bijections between $\Par_\s \cap
\Z^n$ and $\Psi_n$ that can be constructed from $\REM$: for any bijection  
\[ b: \Psi_n \to \Psi_n \] 
the composition of $\REM$ and $b$ is another bijection from $\Par_\s
\cap \Z^n$ to $\Psi_n$

\begin{defn}\label{defn:REMs}
Let $\q = (q_1, \dots, q_n) \in  \Psi_n$. 

\begin{alist}
\itm We define 
\begin{eqnarray*}
	\REM_\s^\q: \Par_\s \cap \Z^n &\to& \Psi_n\\ 
	\x= (x_1, \dots, x_n) &\mapsto& \y=(y_1, \dots, y_n),
\end{eqnarray*}
where \[y_i = x_i + q_i \ \mod s_i.\]
Note that when $\q = (0, \dots, 0),$ the map $\REM_\s^\q$ is the same
as $\REM_\s.$ 
\itm We define 
\begin{eqnarray*}
	\bREM_\s^\q: \Par_\s \cap \Z^n &\to& \Psi_n\\ 
	\x= (x_1, \dots, x_n) &\mapsto& \z=(z_1, \dots, z_n),
\end{eqnarray*}
where \[x_i + z_i = q_i \ \mod s_i.\]
When $\q = (0, \dots, 0),$ we abbreviate $\bREM_\s^\q$ to $\bREM_\s.$
\end{alist}
\end{defn}

\begin{lem}\label{lem:morebij}
Let $\q = (q_1, \dots, q_n) \in  \Psi_n$. 
Then both $\REM_\s^\q$ and $\bREM_\s^\q$ are bijections from $\Par_\s
\cap \Z^n$ to $\Psi_n$. 
\end{lem}

\begin{proof} Let
\begin{eqnarray}
	\Phi_\s^\q: \Psi_n &\to& \Psi_n \label{equ:phi} \\  
	\r= (r_1, \dots, r_n) &\mapsto& \z=(z_1, \dots, z_n), \nonumber
\end{eqnarray}
where \[r_i + z_i = q_i \ \mod s_i.\]
Clearly, $\Phi_\s^\q$ is a bijection and $\bREM_\s^\q = \Phi_\s^\q
\circ \REM.$ Hence, $\bREM_\s^\q$ is a bijection. The proof is similar
for $\REM_\s^\q.$ 
\end{proof}

\section{The anti-lecture hall parallelepiped}\label{sec:anti}
In this section, we will focus on the case when $\s = (n, n-1, \dots,
1).$ For consistency with the terminology in \cite{CorSav2003}, we
call the associated parallelepiped the {\it anti-lecture hall
parallelepiped}. The following theorem is the main result of this section, 
originally proved by Corteel-Lee-Savage \cite[Corollary 4]{CorLeeSav2005}.  

\begin{thm}\label{thm:ehrh-anti}
	The Ehrhart polynomial of the anti-lecture hall polytope
        $P_{(n,n-1,\dots,2,1)}$ is the same as that of the
        $n$-dimensional cube: 
	\[ i(P_{(n,n-1,\dots,2,1)}, t) = (t+1)^n;\]
or equivalently, the $\delta$-vector of $P_{(n,n-1,\dots,2,1)}$ is
given by 
 \[ \delta_{P_{(n,n-1,\dots,2,1)}, i} = A(n,i+1). \] 
\end{thm}

The following lemma is the key ingredient for proving Theorem \ref{thm:ehrh-anti}. 
\begin{lem}\label{lem:sdes-des}
Suppose that $s, s'$ are positive integers and $s - s' = 1.$ Let $r
\in \langle s-1\rangle$ and $r' \in \langle s'-1\rangle$. Then 
\[ \frac{r}{s} > \frac{r'}{s'} \ \Longleftrightarrow \ r > r'.\]
\end{lem}

\begin{proof}First,
\[ \frac{r}{s} > \frac{r'}{s'} \ \Longleftrightarrow \ r s' > r' s
\ \Longleftrightarrow \ (r-r') s' > r'(s-s') = r'. \] 

We then show $r > r'$ if and only if $(r-r') s' > r'.$
Suppose $r > r',$ we have $r - r' \ge 1.$ So $(r-r')s' \ge s' > r'.$ 
Conversely, suppose $r \le r'.$ Then $r - r' \le 0.$ Thus, $(r-r') s'
\le 0 \le r'.$ 
\end{proof}

By Lemma \ref{lem:sdes-des}, one sees that if $\s = (n, n-1, \dots, 2,
1)$, then for any $\r \in \langle n-1\rangle \times \cdots \times \langle
   0\rangle$, $\s$-descents of $\r$ are the same as
regular descents of $\r.$ Hence, we get the following two corollaries
as special cases of Theorem~\ref{thm:REMinv} and
Corollary~\ref{cor:sn=1}. 

\begin{cor}\label{cor:REMinv-anti} Let $\s = (n, n-1, \dots, 2, 1).$
Then the map $\REM_\s$ give a bijection between $\Par_\s \cap \Z^n$
and inversion sequences of length $n.$  
	
Moreover, the inverse of $\REM_\s$ is given by
	\begin{eqnarray*}
   \REM_\s^{-1}: \langle n-1\rangle \times \cdots \times \langle
   0\rangle &\to& \Par_\s \cap \Z^n \\ 
	\r= (r_1, \dots, r_n) &\mapsto& (\des^{<1}(\r) s_1 + r_1,
        \dots, \des^{<n}(\r) s_n + r_n). 
\end{eqnarray*}
\end{cor}

\begin{cor}\label{cor:anti}
	If $\s = (n,n-1,\dots,2,1),$ we have that
	\begin{eqnarray*}
	\L^i(\Par_\s) &=& \{ \x \in \Par_\s \cap \Z^n \ |
        \ \des(\REM_\s(\x)) = i\} \\ 
	&=& \{ \REM_\s^{-1}(\r) \ | \ \des(\r) = i, \ \r \in \langle
                n-1\rangle \times \cdots \times \langle 0\rangle \}, 
	\end{eqnarray*}
	and
	\begin{eqnarray*}
	\l^i(\Par_\s) &=& \# \{ \r \in \langle n-1\rangle \times
        \cdots \times \langle 0\rangle \ | \ \des(\r) = i\} \\ 
		& = & \#\,
        \mathrm{inversion\ sequences\ of\ length}\ n\ \mathrm{that\ 
          have}\ i\ \mathrm{descents}. 
	\end{eqnarray*}
\end{cor}

\begin{proof}[Proof of Theorem \ref{thm:ehrh-anti}]
	The theorem follows from Lemma \ref{lem:invseq}, formula \eqref{equ:delta1}, and Corollary \ref{cor:anti}.
\end{proof}

\section{The reversal of the sequence $\s$}\label{sec:reversal} 

In this section, we assume that $\s = (s_1, \dots, s_n)$ is a sequence
of positive integers and $\u = (u_1, \dots, u_n) =(s_n, \dots, s_1)$
is the reverse of $\s.$
Recall we associate the following set to $\s$:
\begin{equation*}
	\Psi_n  =  \langle s_1-1\rangle \times \cdots \times \langle s_n-1 \rangle.
  \end{equation*}
Similarly, we associate a set to $\u$: 
\begin{equation*}
  \bar{\Psi}_n  = \langle u_1-1\rangle \times \cdots \times \langle u_n-1 \rangle = \langle s_n-1\rangle \times \cdots \times \langle s_1-1 \rangle. 
\end{equation*}
 As usual, we let  
\[ \s^* = (s_1, \dots, s_n, s_{n+1}=1), \text{ and } \u^* = (u_1,
\dots, u_n, u_{n+1}=1).\]

The following lemma suggests a question (Question \ref{ques:bij}),
which is the primary motivation for this section.  
\begin{lem}\label{lem:ehrh-same}
The Ehrhart polynomial of the $\s$-lecture hall polytope $P_\s$ is the
same as the Ehrhart polynomial of the $\u$-lecture hall polytope
$P_\u$; or equivalently, $P_\s$ and $P_\u$ have the same
$\delta$-vectors.  
\end{lem}

\begin{rem}
Note that Theorem \ref{thm:ehrh-anti} and Lemma
\ref{lem:ehrh-same} recover the result on the Ehrhart polynomial of
the lecture hall polytope $P_\s$, where $\s = (1, 2, \dots, n),$ given
in \cite[Corollary 2(i)]{CorLeeSav2005} and \cite[Corollary 1]{SavSch2012}. However, we want to describe a bijection from
the lattice points in the fundamental parallelepiped associated to
$P_{(1,2,\dots,n)}$ to inversion sequences. We will give such a
bijection in Proposition \ref{prop:1-n} in the next section.  
\end{rem}

The proof of Lemma \ref{lem:ehrh-same} is straightforward and is also proved in \cite{CorLeeSav2005}. We defer it to the end of the section. 

The following result follows immediately from
Theorem~\ref{thm:delta-des} and Lemma \ref{lem:ehrh-same}. 

\begin{cor}
For each $i: 0 \le i \le n,$ the two sets
\begin{equation}\label{equ:sset} 
 \{ \r \in \Psi_n\times \langle 0 \rangle  \ | \ \des_{\s^*}(\r) = i \} 
\end{equation}
and
\begin{equation}\label{equ:uset} 
\{ \r \in \bar{\Psi}_n\times \langle 0 \rangle  \ | \ \des_{\u^*}(\r) = i
\}  \end{equation}
have the same cardinality.
\end{cor}

One natural question arises: can we give a simple bijection from
$\Psi_n \times
\langle 0 \rangle$ to $\bar{\Psi}_n \times \langle 0 \rangle$ such that
it induces a 
bijection from the set \eqref{equ:sset} to the set \eqref{equ:uset}
for each $i.$ Note that the last coordinates of any vector in
$\Psi_n\times \langle 
0 \rangle$ or $\bar{\Psi}_n\times \langle 0 \rangle$ is $0,$ which does
not carry any information. For convenience, we drop the last
coordinate when describe the bijection.  
Hence, we rephrase the question as follows:

\begin{ques}\label{ques:bij}
Can we give a simple bijection $b$ from $\Psi_n$ to $\bar{\Psi}_n$ such that
the map $(\r, 0) \mapsto (b(\r),0)$ induces a bijection from the set
\eqref{equ:sset} to the set \eqref{equ:uset} for each $i$?  
\end{ques}

Before discussing Question \ref{ques:bij}, we define a simple
function and fix some notation related to $\s$ and $\u$.

\begin{defn}
For any sequence/vector $\r,$ we denote by $\rev(\r)$ the reverse of
$\r.$ 
\end{defn}

\begin{notn}\label{notn:tstu}
In addition to the usual notation $\s^*$ and $\u^*$, we also define
the following vectors related to $\s$ and $\u:$ 
 \beas \ts & := & (s_0=1, s_1, \dots, s_n, s_{n+1}=1)\\
   \tu & := & (u_0=1, u_1, \dots, u_n, u_{n+1}=1). \eeas
Hence, $\tu$ is the reverse of $\ts.$
\end{notn}

In order to describe a bijection asked by Question \ref{ques:bij}, we
recall the bijection $\Phi_\s^\q$ defined in \eqref{equ:phi}. The map
$\Phi_\s^\0$ is important for this section, so we repeat its
definition here. 
\begin{defn}\label{defn:phi0}
Let $\s = (s_1, \dots, s_n)$ be a sequence of positive integers and
$\0 = (0, \dots, 0).$ Define 
\begin{eqnarray*}
 \Phi_{\s}^{\0}: \Psi_n &\to& \Psi_n \\  
	\r= (r_1, \dots, r_n) &\mapsto& \z=(z_1, \dots, z_n), \nonumber
\end{eqnarray*}
where \[r_i + z_i = 0 \ \mod s_i.\]
For convenience, we abbreviate $\Phi_\s^\0$ to $\Phi_\s.$
\end{defn}

The following theorem is the main result of this section, which
provides a desired bijection to Question \ref{ques:bij}. 

\begin{thm}\label{thm:rev}
For any $\r \in \Psi_n$, we have  
\begin{equation}\label{equ:sdes-udes1}
	\des_{\s^*}(\r, 0) = \des_{\u^*}(\rev(\Phi_\s(\r)), 0).
\end{equation}
\end{thm}

By \eqref{equ:sdes-udes1}, one sees that the map $(\rev \circ\,
\Phi_\s)$ is an answer to Question \ref{ques:bij}. 
If we put all the maps together, we have the following diagram,
denoting by $\pi$ the map that drops the last coordinate of a vector. 
\[\begin{CD}
\Par_{\s^*} \cap \Z^{n+1} @>{\REM_{\s^*}}>> \Psi_n\times \langle 0 \rangle
@>{\pi}>> \Psi_n \\ 
	@. @. @VV{\Phi_\s}V\\
	@. @.\Psi_n\\
	@. @. @VV{\rev}V\\
	\Par_{\u^*} \cap \Z^{n+1} @>{\REM_{\u^*}}>> \bar{\Psi}_n\times \langle 0 \rangle
@>{\pi}>> \bar{\Psi}_n. 
\end{CD}\]
Note that all the maps in the above diagram are bijections. Going
around the diagram from $\Par_{\s^*} \cap \Z^{n+1}$ to $\Par_{\u^*}
\cap \Z^{n+1}$, we obtain a bijection $\Gamma: \Par_{\s^*} \cap
\Z^{n+1} \to \Par_{\u^*} \cap \Z^{n+1}.$ By Theorem \ref{thm:rev} and
Corollary \ref{cor:s*}, we have that $\Gamma$ induces a bijection from
$\L^i(\Par_{\s^*})$ to $\L^i(\Par_{\u^*})$ for each $i.$ 

We can also simplify the above diagram slightly. One checks that
\[ \Phi_\s \circ \pi \circ \REM_{\s^*} = \pi \circ \Phi_{\s^*} \circ
 \REM_{\s*} = \pi \circ \bREM_{\s^*},\] 
where $\bREM_{\s^*}$ is defined in Definition \ref{defn:REMs}. Then we
redraw the diagram: 
	\[\begin{CD}
\Par_{\s^*} \cap \Z^{n+1} @>{\bREM_{\s^*}}>> \Psi_n\times \langle 0
\rangle @>{\pi}>> \Psi_n\\ @. @. @VV{\rev}V\\
\Par_{\u^*} \cap \Z^{n+1} @>{\REM_{\u^*}}>> \bar{\Psi}_n\times \langle 0 \rangle
@>{\pi}>> \bar{\Psi}_n. 
\end{CD}\]
This illustrates that if we use $\bREM_{\s^*}$ for $\Par_{\s^*} \cap
\Z^{n+1}$ and $\REM_{\u^*}$ for $\Par_{\u^*} \cap \Z^{n+1},$ their
image sets have very simple correspondence. 

\begin{cor}
Let $\r \in \Psi_n$. Then for each $i,$ 
\[ \bREM_{\s^*}^{-1}(\r,0) \in \L^i(\Par_{\s^*}) \ \Longleftrightarrow
\ \REM_{\u^*}^{-1}(\rev(\r),0) \in \L^i(\Par_{\u^*}).\]  
\end{cor}
One sees that the bijection $\bREM_{\s^*}$ is useful
sometimes. Despite this, in general we do not have similar results
for $\bREM_{\s*}$ as those for $\REM_\s$ or $\REM_{\s^*}$ described in
Theorem~\ref{thm:REMinv} and Corollary~\ref{cor:s*}. However, we will
show in the next section that $\bREM_{\s^*}$ has a comparable result
for the special cases when $s_1=1.$ 

We need a preliminary lemma before proving Theorem \ref{thm:rev}. The statement of this lemma involves {\it ascents}.

\begin{defn}

Let $\s=(s_1,\dots,s_n)$ be a sequence of positive integers and $\r=(r_1,\dots,r_n) \in \N^n.$ We say that $i$ is an {\it $\s$-ascent of
  $\r$} if $\displaystyle \frac{r_i}{s_i} < \frac{r_{i+1}}{s_{i+1}}.$  
	
We denote by $\Asc_\s(\r)$ the set of $\s$-ascents of $\r,$ and let
$\asc_\s(\r) = \# \Asc_\s(\r)$ be its cardinality. 

When $\s=(1,1,\dots,1),$ we get the {\it (regular) ascents}. We use notation $\Asc(\r)$ and $\asc(\r)$ for this case.
\end{defn}

\begin{lem}
Recall that $\ts$ and $\tu$ are defined in Notation \ref{notn:tstu}.  
For any $\r \in \langle 0 \rangle \times \Psi_n \times \langle 0
\rangle,$ we have   
\begin{equation}\label{equ:sdes-udes}
	\des_\ts(\r) = \asc_\ts(\Phi_{\ts}(\r)) = \des_\tu(\rev(\Phi_\ts(\r))).
\end{equation}
\end{lem}

\begin{proof} Note apply Definition \ref{defn:phi0} to $\ts,$ we have
\begin{eqnarray*}
 \Phi_{\ts}: \langle 0\rangle \times \Psi_n \times \langle 0 \rangle
 &\to& \langle 
 0\rangle \times \Psi_n \times \langle 0 \rangle \\  
	\r= (r_0=0, r_1, \dots, r_n, r_{n+1}=0) &\mapsto& \z=(z_0=0,
        z_1, \dots, z_n, z_{n+1}=0), \nonumber 
\end{eqnarray*}
where \[r_i + z_i = 0 \ \mod s_i.\]

Let $\r = (r_0, r_1, \dots, r_n, r_{n+1}) \in \langle 0\rangle \times
\Psi_n \times
\langle 0 \rangle$ and $\z = (z_0, z_1, \dots, z_n, z_{n+1}) =
\Phi_{\ts}(\r).$ 
	By the definition of $\Phi_{\ts},$ we have that 
	\begin{equation}\label{equ:ri2zi}
		z_i = \begin{cases} s_i - r_i, & \text{if $r_i \neq 0$}; \\
			0, & \text{if $r_i = 0$}.\end{cases}
	\end{equation}
One can verify that the following four statements are true for $i: 0
\le i \le n,$ by using \eqref{equ:ri2zi}. 
	\begin{ilist}
	\itm Suppose $r_i \neq 0$ and $r_{i+1} \neq 0.$ Then $i$ is an
        $\ts$-descent of $\r$ if and only if $i$ is an $\ts$-ascent of
        $\z.$ 
	\itm Suppose $r_i \neq 0$ and $r_{i+1} = 0.$ Then $i$ is an
        $\ts$-descent of $\r$ and $i$ is not an $\ts$-ascent of $\z.$ 
	\itm Suppose $r_i = 0$ and $r_{i+1} \neq 0.$ Then $i$ is not
        an $\ts$-descent of $\r$ and $i$ is an $\ts$-ascent of $\z.$ 
	\itm Suppose $r_i = 0$ and $r_{i+1} = 0.$ Then $i$ is not an
        $\ts$-descent of $\r$ and $i$ is not an $\ts$-ascent of $\z.$ 
	\end{ilist}
However, since $r_0=r_{n+1}=0,$ we see that the number of occurrences
of situation (ii) and the number of occurrences of situation (iii) are
the same. Therefore, the first equality in \eqref{equ:sdes-udes}
follows. The second equality in \eqref{equ:sdes-udes} follows from the
first one trivially.  
\end{proof}

\begin{proof}[Proof of Theorem \ref{thm:rev}]
	One verifies that
	\begin{eqnarray*}
		\des_{\s^*}(\r, 0) &=& \des_{\ts}(0, \r, 0) = \des_{\tu}(\rev(\Phi_{\ts}(0, \r, 0))) \\
		&=& \des_{\tu}(0, \rev(\Phi_{\s}(\r)), 0) = \des_{\u^*}(\rev(\Phi_\s(\r)), 0),
	\end{eqnarray*}
where the first and last equalities follow from the fact that
appending $0$'s at the beginning of a nonnegative-entry vector does
not create descents, the second equality follows from
\eqref{equ:sdes-udes}, and the third equality follows from the
definitions of $\Phi_\ts$ and $\Phi_\s.$  
\end{proof}

Finally, We prove Lemma \ref{lem:ehrh-same}.
\begin{proof}[Proof of Lemma \ref{lem:ehrh-same}]
Note that 
\begin{eqnarray*}& & 0 \le \frac{x_1}{s_1} \le \frac{x_2}{s_2} \le \cdots \le \frac{x_n}{s_n} \le 1 \\ 
	&\Longleftrightarrow& 0 \le \frac{s_n-x_n}{s_n} \le \frac{s_{n-1}-x_{n-1}}{s_{n-1}} \le \cdots \le \frac{s_1-x_1}{s_1} \le 1. 
\end{eqnarray*}
Hence, one see that the map $\x \mapsto \rev(\s-\x)$ gives a affine
transformation from $P_\s$ to $P_\u.$ Moreover, it is easy to see the
transformation is unimodular. The desired result follows. 
\end{proof}

\section{The case when $s_1=1$}\label{sec:s1=1} 



In this section, we focus on the special case when the first entry of
$\s$ is $1.$
\begin{lem}
Let $\r \in \Psi_n$. If $s_1=1,$ we have
\begin{equation}\label{equ:sdes-sasc}
	\des_{\s^*}(\r, 0) = \asc_{\s^*}(\Phi_\s(\r),0)= \asc_{\s}(\Phi_\s(\r)).
\end{equation}
\end{lem}
\begin{proof}
Similarly to the proof of Theorem \ref{thm:rev}, we have by
\eqref{equ:sdes-udes} that
\begin{equation}\label{equ:scac}
	\des_{\s^*}(\r, 0) = \des_{\ts}(0, \r, 0) =
\asc_{\ts}(\Phi_{\ts}(0, \r, 0)) = \asc_{\ts}(0, \Phi_\s(\r), 0),
\end{equation}
where $\ts$ is defined in Notation \ref{notn:tstu}.

Suppose $s_n =1.$ Then the first two entries of the vector $(0,
\Phi_\s(\r),0)$ are $(0,0)$, which is not an ascent. Hence, 
\[ \asc_{\ts}(0, \Phi_\s(\r),0) = \asc_{\s^*}(\Phi_\s(\r),0)=
\asc_{\s}(\Phi_\s(\r)).\] 
Therefore equation \eqref{equ:sdes-sasc} follows.
\end{proof}

\begin{cor}\label{cor:s1=1}
Suppose $s_1 = 1.$ 
(Recall $\bREM_\s$ is defined in part b) of Definition
\ref{defn:REMs}.) Then
\begin{eqnarray*}
	\L^i(\Par_{\s^*}) &=&  \{ \x \in \Par_{\s^*} \cap \Z^{n+1} \ |
        \ \asc_{\s^*}(\bREM_{\s^*}(\x)) = i\} \\ 
	&=& \{ \bREM_{\s^*}^{-1}(\r) \ | \ \asc_{\s^*}(\r) = i, \ \r
        \in \Psi_n \times \langle 0 \rangle \} \\ 
	&=& \{ \bREM_{\s^*}^{-1}(\r,0) \ | \ \asc_{\s}(\r) = i, \ \r
        \in \Psi_n
        \}, 
\end{eqnarray*}
and
\begin{eqnarray*}
	\l^i(\Par_{\s^*}) &=& \# \{ \r \in \Psi_n\times \langle 0 \rangle
        \ | 
        \ \asc_{\s^*}(\r) = i\} \\ 
	&=& \# \{ \r \in \Psi_n \ | \ \asc_{\s}(\r) = i\}.  
\end{eqnarray*}
Hence, if let $\pi$ be the map that drops the last coordinate of a
vector, we have that $\pi \circ \bREM_{\s^*}$ gives a bijection
between $\Par_{\s*} \cap \Z^{n+1}$ and $\Psi_n$ such that 
\[ \x \in \L^i(\Par_{\s^*}) \ \Longleftrightarrow
\ \asc_\s(\pi(\bREM_{\s^*}(\x))) = i.\] 
\end{cor}

\begin{proof}
It follows from \eqref{equ:sdes-sasc} and Corollary \ref{cor:s*}.
\end{proof}
Therefore we can describe the $\delta$-vector of $\s$-lecture hall
polytope with $\s$-ascents when $s_1=1.$ 

\begin{cor}\label{delta-asc}
Suppose $s_1=1.$ Then the $\delta$-vector of $P_\s$ is given by
\begin{equation}\label{equ:delta-asc}
\delta_{P_\s,i} = \# \{ \r \in \Psi_n \ | \ \asc_{\s}(\r) = i\},\ 0 \le i
\le n. 
\end{equation}
\end{cor}

Corollary \ref{delta-asc} extends easily to arbitrary $s_1$ using equation~\eqref{equ:scac}. This result appears in \cite{SavSch2012} (a special case of their Theorem~5), but we have no need to state it here.

We find it is interesting to compare Corollary~\ref{cor:sn=1} and
Corollary~\ref{cor:s1=1}, and equations~\eqref{equ:delta-des1} and
\eqref{equ:delta-asc}. These are parallel results for the
cases $s_n=1$ and $s_1=1.$ One sees that the result of the case
$s_n=1$ is much easier to obtain than that of the case  $s_1=1.$ 
The above two corollaries also tell us that when $s_1=1,$ it is better
to use the map
$\bREM_{\s^*}$ than $\REM_{\s^*}$. 

Finally, applying the above results to $\s =(1,2,\dots,n),$ we obtain
a bijection from the lattice points in the fundamental parallelepiped
$\Par_{\s*}$ associated to $P_{(1,2,\dots,n)}$ to inversion
sequences. 

\begin{prop}\label{prop:1-n}
	Suppose $\s = (1, 2, \dots, n).$ Let $\pi$ be the map that
        drops the last coordinate of a vector. Then the composition
        map  
\begin{eqnarray*}
	\L^i(\Par_{\s^*})
        &\stackrel{\bREM_{\s^*}}{\xrightarrow{\hspace*{2cm}}}& \Psi_n
        \times \langle 0 \rangle \\ 
	&\stackrel{\pi}{\xrightarrow{\hspace*{2cm}}}& \Psi_n \\
	&\stackrel{\rev}{\xrightarrow{\hspace*{2cm}}}& \bar{\Psi}_n =
        \langle n-1 \rangle \times \cdots \times \langle 0 \rangle  
\end{eqnarray*}
gives a bijection from $\Par_{\s^*} \cap \Z^{n+1}$ to the inversion
sequences of length $n.$ Furthermore, for any $\x \in \Par_\s \cap
\Z^{n+1},$ we have 
	\begin{eqnarray}		
 \text{ the last coordinate of $\x$} &=&
 \asc(\pi(\bREM_{\s^*}(\x))) \label{equ:n-1_1}\\ 
		&=& \des(\rev(\pi(\bREM_{\s^*}(\x)))). \label{equ:n-1_2}
	\end{eqnarray}
\end{prop}
\begin{proof}
The only thing we need to verify is the equality
\eqref{equ:n-1_1}. (Note that the equality \eqref{equ:n-1_2} follows
from the equality \eqref{equ:n-1_1} easily.) By Corollary
\ref{cor:s1=1}, we have 
 \[\text{ the last coordinate of $\x$} =
 \asc_{\s}(\pi(\bREM_{\s^*}(\x))).\]  
However, since $\s=(1,2,\dots,n)$, we have by Lemma \ref{lem:sdes-des}
that for any $\r \in \langle 0 \rangle \times \langle 1 \rangle \times \cdots \times \langle n-1 \rangle $,
an $\s$-ascent of $\r$ is the same as a regular ascent of $\r$, and
vice versa. Hence equation~\eqref{equ:n-1_1} follows. 
\end{proof}

\end{document}